\documentclass[11pt,reqno]{amsart}

\usepackage[letterpaper]{geometry}
\usepackage{graphicx}

\theoremstyle{plain}
\newtheorem{theorem}{Theorem}
\newtheorem{lemma}{Lemma}
\newtheorem{corollary}[theorem]{Corollary}

\theoremstyle{definition}
\newtheorem{remark}{Remark}

\numberwithin{equation}{section}
\numberwithin{theorem}{section}
\numberwithin{table}{section}
\numberwithin{figure}{section}

\newcommand{\G}{\mathcal G}

\begin{document}

\title[On Hofstadter's $G$--Sequence]
{A Combinatorial interpretation of Hofstadter's G-sequence}

\author{Mustazee Rahman}

\address{Department of Mathematics\\
University of Toronto\\
40 St. George Street\\
ON M5S 2E4\\
Canada}

\email{mustazee.rahman@utoronto.ca}

\date{October 2, 2010}




\begin{abstract}We give a combinatorial interpretation of a classical meta-Fibonacci sequence
defined by $G(n) = n - G(G(n-1))$ with the initial condition $G(1) = 1$, which appears in Hofstadter's
``G\"odel, Escher, Bach: An Eternal Golden Braid''. The interpretation is in terms of an infinite labelled tree.
We then show a couple of corollaries about the behaviour of the sequence $G(n)$ directly from the interpretation.
\end{abstract}

\maketitle

\section{Introduction} \label{Sec:intro}

In his book ``G\"odel, Escher, Bach: An Eternal Golden Braid" Douglas Hofstadter
introduced his $G$-sequence defined as
\begin{equation} \label{eqn:g} G(n) = n - G(G(n-1))\;; \quad G(1) = 1.\end{equation}
This recursion is part of the general family of recursions given by $G(n) = n - G(G^k(n-1))$ with initial condition
$G(1) = 1$. The superscript of $k \geq 1$ means a $k$-fold composition of the function $G(n)$.  Recursions of 
this form, where the argument of the defining terms depend on previous values of the recursive function, are
called meta-Fibonacci or nested recursions. There is knowledge about Hofstadter's $G$-sequence in literature
nowadays, but little is known about the other $k$-fold recursions above.

Let $F_n$ denote the Fibonacci numbers, defined by $F_n = F_{n-1} + F_{n-2}$ and $F_1 = F_2 = 1$. Meek and Van Rees
\cite{M&R} showed that if $n = F_{r_1} + \cdots + F_{r_j}$ is the Zeckendorf representation\footnote{Given any positive
integer $n$, it is possible to write $n$ uniquely as $n = F_{r_1} + \cdots + F_{r_j}$ where $r_i \geq r_{i+1} + 2$ for $1 \leq i \leq j-1$.
This is called the Zeckendorf representation of $n$.} of $n$ then $G(n) = F_{r_1 -1} + \cdots + F_{r_j - 1}$. Soon both Granville
and Rasson \cite{G&R}, and Downey and Griswold \cite{D&G} showed that $G(n) = \lfloor (n+1)\phi^{-1} \rfloor$ where
$\phi = \frac{\sqrt{5} + 1}{2}$ is the golden ratio. Our result is an interpretation for $G(n)$ in terms of counting labels in an infinite
labelled tree. After our discovery we learned that the result was known by some in meta-Fibonacci circles, but to the best of our
knowledge there is no published proof of this combinatorial interpretation in literature. The purpose of this paper is not only to
give a proof of the interpretation, but also to provide motivation for adapting this approach in order to find a combinatorial
interpretation of the related $k$-fold recursion above.

\section*{Acknowledgements}
The author would like to thank professor Steve Tanny for bringing the questions addressed in this paper to his attention
during an undergraduate project. Without his guidance most of this work would not be possible. The author is also grateful
to Ilia Smirnov for carefully proofreading the paper.

\section{Combinatorial interpretation of the recursion $G(n)$} \label{Sec:g}
Throughout this section we will refer to Hofstadter's $G$-sequence as defined in \eqref{eqn:g} as $G(n)$. Table \ref{tbl:G}
contains the first 20 values of $G(n)$. Notice from this table that the difference between a term of $G(n)$
from the previous term is always 0 or 1. Such a sequence of positive integers is called slow-growing. It is not difficult to
show via induction that $G(n)$ remains slow-growing for all $n$. More importantly, let $f(n) = \#\, G^{-1}(\{n\})$ for $n \geq 1$,
which we call the frequency sequence of $G(n)$. The frequency sequence displays many patterns (see Table \ref{tbl:F}), for
example, it consists only of 1s and 2s. In fact, let us interpret the frequency sequence as an infinite word $W = \prod_{n=1}^{\infty}f(n)$
where the product stands for concatenation. Setting $w_1 = 2, w_2=1, w_3 = 2$, and $w_n = w_{n-1}w_{n-2}$ for $n \geq 4$,
we can experimentally verify that the initial segments of the word $W$ factorizes as $W = w_1w_2w_3^2w_4^2\cdots w_n^2$.
Such observations about the frequency sequence motivate us to define a labelled tree such that for each $n$, $G(n)$ is the label
of the parent of vertex $n+1$.

\begin{table}[!ht]
\fontsize{10}{10}\selectfont
\caption{First 20 terms of $G(n)$}\label{tbl:G}
\begin{tabular}{|r | r r r r r r r r r r r r r r r r r r r r|}
\hline
$n$&1&2&3&4&5&6&7&8&9&10&11&12&13&14&15&16&17&18&19&20\\\hline
$G(n)$&1&1&2&3&3&4&4&5&6&6&7&8&8&9&9&10&11&11&12&12\\\hline
\end{tabular}
\end{table}

\begin{table}[!ht]
\fontsize{10}{10}\selectfont
\caption{First 20 terms of the frequency sequence $f(n)$}\label{tbl:F}
\begin{tabular}{|r | r r r r r r r r r r r r r r r r r r r r|}
\hline
$n$&1&2&3&4&5&6&7&8&9&10&11&12&13&14&15&16&17&18&19&20\\\hline
$f(n)$&2&1&2&2&1&2&1&2&2&1&2&2&1&2&1&2&2&1&2&1\\\hline
\end{tabular}
\end{table}

This tree, denoted as $\G$, is defined recursively as follows. $\G$ has a root vertex whose left subtree is a copy of $\G$.
The root also has a right child which itself has a copy of $\G$ as its only subtree. The definition is pictured in the left-half of
Figure \ref{fig:G}. The labeling procedure works by denoting the root of $\G$ as vertex $1$, and then labeling the vertices in
increasing order of their height from the root. All vertices at a specific height are labelled in increasing order from \emph{right to left}.
The right half of Figure \ref{fig:G} shows $\G$ labeled up to height four.

\begin{figure}[htpb]
\begin{center}
\includegraphics[scale=.35]{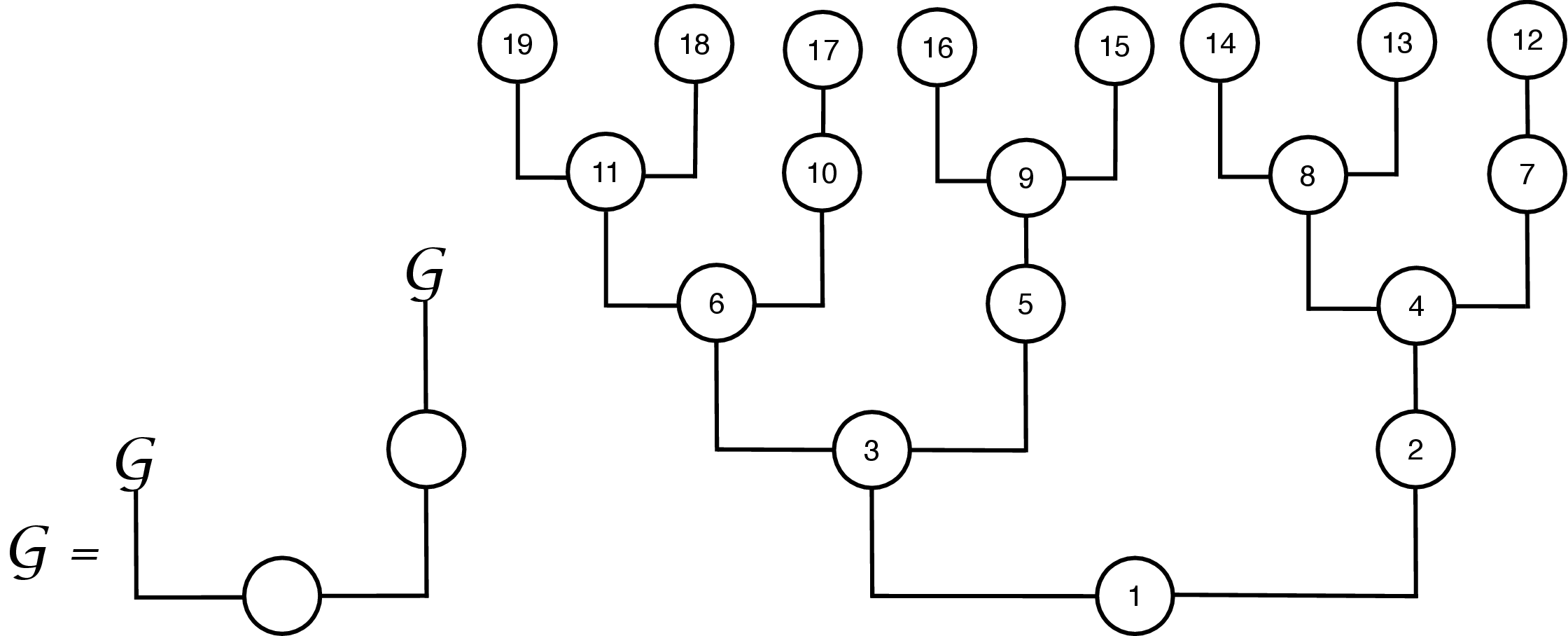}
\caption{$\G$ is defined on the left. To the right is $\G$ labelled up to height 4.} \label{fig:G}
\end{center}
\end{figure}

\begin{lemma} \label{lem}
The tree $\G$ contains $F_{h+2}$ vertices at height $h$. Consequently, the number of vertices in $\G$ from
height 0 through $h$ is $F_{h+4} - 2$.\end{lemma}

\begin{proof} From the definition of $\G$ we note that vertices at height $h$ of $\G$ consists of vertices at height $h-2$ of $\G$
contributed by the right subtree of the root and vertices at height $h-1$ contributed by the left subtree of the root. As such, the
number of vertices at height $h$ satisfies the Fibonacci recursion with initial conditions $F_2$ and $F_3$ for heights 0 and 1
respectively. It follows that there are $F_{h+2}$ vertices at height $h$ and that the number of vertices from height $0$ to $h$
is $F_2 + \cdots + F_{h+2} = F_{h+4} - 2$. We deduce the latter from the well-known Fibonacci identity $F_1 + \cdots + F_h = F_{h+2} - 1$.
See, for example, \cite[p.\ 2]{ABJQ} where a combinatorial proof is provided.\end{proof}

\begin{remark} We could alter the tree $\G$ to make it so that $G(n)$ is the label of the parent of vertex $n$. To do this we need
to insert a new vertex below the current root of $\G$ and then attach it to the root. We then redo the labeling starting from the new
vertex. However, there are certain advantages with the current definition. First, we think that it makes the recursive definition of $\G$
more pleasant. Further, the definition as it stands will make some of the computations in the proof of the combinatorial interpretation
easier and the exposition more clear.\end{remark}

\begin{theorem} \label{thm:G}
Let $g(n)$ denote the label of the parent of vertex $n+1$ in the infinite tree $\mathcal{G}$. 
Then $g(n)$ satisfies the recursion $g(n) = n - g(g(n-1))$ with the initial value $g(1) = 1$.\end{theorem}

\begin{proof} The proof is by induction on the vertex labels $n$. By comparing the values of $G(n)$ from Table \ref{tbl:G} with the values of
$g(n)$ from Figure \ref{fig:G}, we see that the theorem holds for the first 19 vertices, which consists of all vertices of $\G$ up to height 4.
Now suppose that vertex $n+1$ is located at height $h > 4$ and that the theorem holds for all vertices in $\G$ preceding $n+1$.

As $h > 4$, vertex $n+1$ is either located on the left subtree of vertex 1 or on the subtree emanating from vertex 2.
Both these subtrees are a copy of $\G$. Denote the subtree containing vertex $n+1$ as $\G'$, and the subtree not containing
$n+1$ as $\G''$. If we now remove all the labels of $\G'$ that it inherits from $\G$, and relabel $\G'$ in the same manner as
we labeled $\G$ starting from 1, then vertex $n+1$ of $\G$ will receive a new label $n_0 + 1$ on $\G'$ with $n_0 < n$ since
$n_0 + 1$ is located at a lower height in $\G'$. (Specifically, at height $h-1$ of $\G'$ if $\G'$ is the left subtree of
vertex 1 or at height $h-2$ of $\G'$ if it is the subtree of vertex 2.) As an example, consider $n + 1 = 10$ in which case
$\G'$ is the left subtree of vertex 1 and $n_0 + 1 = 5$. We establish the recursive formula at index $n$ through its validity
at index $n_0$ and by evaluating the differences $(n+1) - (n_0 + 1)$, $g(n) - g(n_0)$, and $g(g(n-1)) - g(g(n_0-1))$.\\

\paragraph{\textbf{Evaluating $(n+1) - (n_0 + 1)$:}} The difference $(n+1) - (n_0+1)$ is the number of vertices in $\G$ consisting of vertices 1, 2, and all the vertices preceding $n+1$ that are in $\G''$. Such vertices preceding $n+1$ in $\G''$ consist of all vertices from height 0 to $h-2$ of $\G''$. As $\G''$ is isomorphic to $\G$, there are $F_{h+2} - 2$ vertices from height 0 to $h-2$ by Lemma \ref{lem}. Therefore, \begin{equation} \label{eqn1} n+1 = (n_0 + 1) + 2 + (F_{h+2} - 2). \end{equation}

\paragraph{\textbf{Evaluating $g(n) - g(n_0)$:}} To this end notice that the parent of vertex $n+1$ in $\G$ is the same as the parent of the vertex labelled $n_0+1$ in $\G'$, but its label differs between the two trees. Vertices 1 or 2 cannot be the parent as $n+1$ is above height 4. The difference $g(n) - g(n_0)$ is again given by accounting for vertices 1 and 2 in $\G$ along with the vertices in $\G$ preceding the parent of $n+1$ and residing in $\G''$. Since the parent of vertex $n+1$ is at height $h-1$ of $\G$, the vertices in $\G''$ that precede the parent of $n+1$ range from height 0 to $h-3$ of $\G''$. Counting as before we deduce that $g(n) - g(n_0) = F_{h+1}$. Thus, since $n = n_0 + F_{h+2}$ from ~(\ref{eqn1}), we conclude that
\begin{equation} \label{eqn2} n - g(n) = n_0 - g(n_0) + F_h.\end{equation}

\paragraph{\textbf{Evaluating $g(g(n-1)) - g(g(n_0-1))$:}}  We begin by showing that except for some instances vertex $g(g(n-1))$ in $\G$ is the same
as the one labeled $g(g(n_0-1)$ in the separate labeling of $\G'$. In this case we will be able to compute $g(g(n1)) - g(g(n_0-1))$ readily.
Let us assume that vertex $n+1$, considered as a vertex of $\G$ situated in the subtree $\G'$, is not the first or last vertex at its height in $\G'$.

A consequence of vertex $n+1$ not being the first vertex at it height in $\G'$ is that vertex $n$ lies in $\G'$ and points to the same vertex as the one labelled $n_0$ in the separate labeling of $\G'$. As such $g(n-1)$ points to the parent vertex of $n_0$ in $\G'$, which is labelled $g(n_0-1)$ in the separate labeling of $\G'$. This parent cannot be the last vertex at its height in $\G'$, for otherwise, it would be the common parent of vertices $n_0$ and $n_0 +1$ in $\G'$ located at the height above. This would make $n_0+1$ the last vertex at its height in $\G'$, contradicting the assumption that vertex $n+1$ in $\G$, pointing to the same vertex as the one separately labeled $n_0+1$ in $\G'$, is not so. 

As vertex $g(n-1)$ of $\G$ is not the last vertex at its height in the subtree $\G'$, vertex $g(n-1) + 1$ of $\G$ lies in the subtree $\G'$
and points to the same vertex as the one separately labelled $g(n_0-1) + 1$ in $\G'$. Hence the parent of vertex $g(n-1) + 1$ in $\G$,
lying at height $h-2$ of $\G$, also lies in $\G'$ and points to vertex $g(g(n_0-1))$ in the separate labeling of $\G'$ (recall that $h > 4$
so vertex $g(g(n_0-1))$ is defined in $\G'$). This confirms that vertex $g(g(n-1))$ in $\G$ is the same as the vertex labeled $g(g(n_0-1))$
in $\G'$.

The difference $g(g(n-1)) - g(g(n_0-1))$ in labels again comes from vertices 1, 2 and all vertices in $\G$ preceding $g(g(n-1))$
that lie on $\G''$. Since vertex $g(g(n-1))$ is at height $h-2$ of $\G$, the vertices preceding $g(g(n-1))$ that lie on $\G''$ range from
height 0 to $h-4$ of that subtree. There are $F_h - 2$ of them by Lemma \ref{lem} and thus
\[g(g(n-1)) = g(g(n_0-1)) + 2 + (F_h - 2) = g(g(n_0-1)) + F_h\,.\] By the induction hypothesis $n_0 - g(n_0) = g(g(n_0 -1))$, and so
by (\ref{eqn2}) we get \[n - g(n) = n_0 - g(n_0) + F_h = g(g(n_0-1)) + F_h = g(g(n-1))\,.\]

We now deal with the exceptional cases, assuming first that vertex $n+1$ in $\G$ is located on the left subtree of vertex 1 and
at heigh $h > 4$. Deviating from the previous notation a bit, denote the left subtree containing $n+1$ as $\G_l$ and the
subtree emanating from vertex 2 as $\G_r$. Vertex $n+1$ is at height $h-1$ of $\G_l$. If it is the first vertex at height $h-1$ of
$\G_l$ then vertex $n$ is the last vertex at height $h-2$ of $\G_r$. So $g(n-1)$ is the last vertex at height $h-3$ of $\G_r$ and
so $g(n-1) + 1$ is the first vertex at height $h-2$ of $\G_l$. Vertex $g(g(n-1))$ is then the first one at height $h-3$ of $\G_l$.

Meanwhile, when $\G_l$ is labeled separately as before, vertex $n+1$ in $\G$ will again point to the same vertex as some vertex
labeled $n_0+1$ in $\G_l$. Then the vertex labeled $n_0$ in $\G_l$ becomes the last vertex at height $h-2$ of $\G_l$, which implies
that $g(n_0-1)$ is the last vertex at height $h-3$ of $\G_l$, and $g(n_0-1) + 1$ is thus the first vertex of $\G_l$ at height $h-2$. So vertex
$g(g(n_0-1))$ in $\G_l$, the parent of $g(n_0-1) + 1$, is the first vertex at height $h-3$ of $\G_l$. It points to the same vertex as the one
labelled $g(g(n-1))$ in $\G$. The difference $g(g(n-1)) - g(g(n_0-1))$ is thus $F_h$ due to the same reasons as in the previous case.
Thus from the induction hypothesis for $n_0$, we deduce as before that $n-g(n) =g(g(n-1))$.

When vertex $n+1$ in $\G$ is the last vertex at height at height $h-1$ of $\G_l$, vertex $g(n-1)$ is the common parent of vertices
$n$ and $n+1$, situated as the last vertex at height $h-2$ of $\G_l$. So $g(n-1) + 1$ is the first vertex at height $h-2$ of $\G_r$
making $g(g(n-1))$ the first vertex at height $h-3$ of $\G_r$. On the other hand, with $n_0 +1$ defined as before, vertex $g(g(n_0-1))$
in $\G_l$ will be the first one at height $h-2$ of $\G_l$. Both vertices $g(g(n-1))$ and $g(g(n_0-1))$ are located at height $h-1$ of $\G$.

One must be careful about calculating $g(g(n-1)) - g(g(n_0-1))$ because vertex $g(g(n-1))$ of $\G$ occurs \emph{before} the vertex
pointing to $g(g(n_0-1))$ in $\G$. Although the label of the latter vertex on $\G_l$, which is $g(g(n_0-1))$, is less than $g(g(n-1))$.
In this situation we simply note that since vertex $g(g(n_0-1))$ is the first one at height $h-2$ of $\G_l$, its label is
$F_2 + \cdots + F_{h-1} + 1 = F_{h+1} - 1$. Similarly, as $g(g(n-1))$ is the first vertex at height $h-1$ of $\G$, its label is $F_{h+2} - 1$.
Therefore $g(g(n-1)) - g(g(n_0-1)) = F_{h+2} - F_{h+1} = F_h$. As before it follows from the induction hypothesis on $n_0$ that
$n - g(n) = g(g(n-1))$.

When vertex $n+1$ is located on the subtree $\G_r$, the exceptional cases are analogous to the previous ones. The case
for vertex $n+1$ being the last one at height $h-2$ of $\G_r$ is analogous to the case when $n+1$ is the first vertex at height
$h-1$ of $\G_l$. When vertex $n+1$ is the first vertex at height $h-2$ of $\G_r$ the situation is analogous to when $n+1$ is the
last vertex at height $h-1$ of $\G_l$. This completes our induction and the proof of the theorem.
\end{proof}

We now present a corollary which is the simplest case of Meek and Van Rees's result on how $G(n)$ acts on positive
integers when written in their Zeckendorf representation. Recall from the introduction that the authors showed if $n = F_{r_1} + \cdots + F_{r_j}$
is the Zeckendorf representation of $n$ then $G(n) = F_{r_1-1} + \cdots + F_{r_j-1}$. We now show this for $n = F_r$, that is,
for the Fibonacci numbers. With more work one could get the whole result from the combinatorial interpretation but we will not
pursue that direction. Instead we will prove the factorization property of the word $\prod_{n=1}^{\infty} f(n)$ that we observed earlier
experimentally, and which led to our discovery of the combinatorial interpretation for $G(n)$.

\begin{corollary} \label{cor1}
$g(F_n) = F_{n-1}$ for $n \geq 2$, where $F_n$ is the $n^{th}$ Fibonacci number.
\end{corollary}

\begin{proof}For $n > 4$, the second vertex at height $n-3$ is labeled $F_2 + \cdots + F_{n-2} + 2 = F_n$. Its parent is the first
vertex at height $n-4$ with label $F_{n-1}-1$. Thus the parent of vertex $F_n + 1$ is the second vertex at height $n-4$ since the first
one has vertices $F_n-1$ and $F_n$ as its children. Hence $g(F_n) = F_{n-1}$. For $2 \leq n \leq 4$, one can verify the claim from
Figure \ref{fig:G}. \end{proof}

\begin{corollary} \label{cor2}
Define words $w_1 = 2, w_2 = 1, w_3 = 2 = w_1$ and $w_n = w_{n-2}w_{n-1}$ for $n>3$. Let $f(n)$ denote the frequency sequence of $g(n)$.
The infinite word $W = \prod_{n=1}^{\infty} f(n)$ factorizes as $W = w_1w_2\prod_{n=3}^{\infty}w_n^2$.\end{corollary}

\begin{proof}The combinatorial interpretation makes it clear that $g(n)$ is slow-growing, and Theorem \ref{thm:G} implies that
the frequency sequence of $G(n)$ from \eqref{eqn:g} is the same as that of $g(n)$. We note that $f(n)$ is the number of children
of vertex $n$ in $\G$. Indeed, if vertex $n$ has children labeled from $a$ to $b$ then $g(k) = n$ precisely when $k$ ranges from
$a+1$ to $b+1$. Let $W_h$ denote the word produced by concatenating from left to right the values of $f(n)$
as $n$ ranges in increasing order over the vertices at height $h$ of $\G$. For $h>1$, the vertices at height $h$ of $\G$ are arranged by
placing the vertices at height $h-2$ of $\G$ to the right of the vertices at height $h-1$ of $\G$, following from the recursive definition
of $\G$. Thus $W_h = W_{h-2}W_{h-1}$ with $W_0 = 2$ and $W_1 = 12$. We show that $W_h = w_{h+1}w_{h+2}$ for $h > 0$.
We have that $W_1 = 12 = w_2w_3$ and $W_2 = 212 = w_3w_4$. Assuming that $W_h = w_{h+1}w_{h+2}$ for $1 \leq h < N$, we get that
$W_N = W_{N-2}W_{N-1} = w_{N-1}w_N w_N w_{N+1} = w_{N+1}w_{N+2}$ where the last equality uses the recursive definition of $w_n$.
The claim follows by induction. Finally, to finish the proof we note that $W_0 = 2 = w_1$ and so\begin{equation*}
\prod_{n=1}^{\infty}f(n) = \prod_{h=0}^{\infty} W_h = W_0 \prod_{h=1}^{\infty} W_h
= w_1 \prod_{h=1}^{\infty}w_{h+1}w_{h+2} = w_1w_2 \prod_{h=3}^{\infty}w_h^2.\end{equation*}
\end{proof}

The consequence of Corollary \ref{cor2} is that it allows us to view the word $W$ as two intertwined copies
of itself along with the initial seeds $w_1$ and $w_2$. More precisely, the above factorization shows that
$W = w_1w_2\prod_{n=3}^{\infty}w_n^2 = w_3 \prod_{n=2}^{\infty}w_nw_{n+1} = w_3\prod_{n=2}^{\infty}w_{n+2} = \prod_{n=3}^{\infty}w_n$.
What we hope is that one can find analogous factorizations of the frequency sequences generated by the $k$-fold recursions that was discussed in the introduction. In this manner one may find corresponding infinite trees for the $k$-fold recursions and proceed to derive a combinatorial interpretation.


\begin{thebibliography}{99}

\bibitem{ABJQ} A.T. Benjamin and J.J. Quinn, Proofs that Really Count, The Mathematical Association of America, 2003.

\bibitem{D&G} P.J. Downey and R.E. Griswold, On a family of nested recursions, Fibonacci Quarterly 22 (1984), no. 4, 310--317.

\bibitem{G&R} V. Granville and J-P Rasson, A strange recursive relation, J. Number Theory 30 (1988), no. 2, 238--241.

\bibitem{Hofstadter} D. R. Hofstadter, G\"odel, Escher, Bach: An Eternal Golden Braid, Random House, 1979.

\bibitem{M&R} D.S. Meek and G.H. Van Rees, The solution of an iterated recurrence, Fibonacci Quarterly 22 (1984), no. 2, 101--104.

\end{thebibliography}
\end{document}